\documentclass[preprint,11pt]{elsarticle}
\usepackage[a4paper]{geometry}
\usepackage{xcolor}

\usepackage{booktabs}
\usepackage{makecell}
\usepackage{scrextend}
\usepackage{pifont}

\usepackage{lineno,hyperref}
\usepackage{amsmath,amssymb,amsthm,bbm,bm,xargs,cleveref}
\usepackage{mathtools}

\modulolinenumbers[5]

\definecolor{darkblue}{HTML}{0000FF}

\newtheorem{Co}{Corollary}
\newtheorem{Lem}{Lemma}
\newtheorem{Prop}{Proposition}
\newtheorem{Thm}{Theorem}

\theoremstyle{remark} 
\newtheorem{Def}{Definition}

\newtheorem{Rem}{Remark}

\def\esssup{\operatornamewithlimits{esssup}}

\renewcommand{\geq}{\geqslant}
\renewcommand{\leq}{\leqslant}

\newcommand{\yes}{\ding{51}}
\newcommand{\no}{\ding{55}}

\newcommand{\1}{\mathbbm 1}

\newcommand{\R}{\mathbb R}

\newcommand{\bgamma}{{\boldsymbol{\gamma}}}
\newcommand{\balpha}{{\boldsymbol{\alpha}}}

\newcommand{\nset}{\mathbb{N}}

\newcommand{\requ}{\sigma}

\newcommand{\splinespace}{\mathcal{S}}
\newcommand{\lbeta}{\lfloor\beta\rfloor}

\newcommand{\A}{\mathcal A}

\renewcommand{\H}{\mathcal H}

\renewcommand{\L}{\mathcal L}

\newcommand{\NN}{\mathsf{NN}}

\newcommand{\W}{\mathcal W}

\newcommand{\Z}{\mathcal Z}

\newcommand{\rmd}{\mathrm{d}}

\newcommand\eps\varepsilon
\newcommand{\rset}{\mathbb{R}}
\begin{document}

\begin{frontmatter}

\title{Simultaneous approximation of a smooth function\\and its derivatives by deep neural networks\\ with piecewise-polynomial activations}

\author[a]{Denis Belomestny\corref{corr}}
\cortext[corr]{Corresponding author}
\ead{denis.belomestny@uni-due.de}

\author[b]{Alexey Naumov}
\ead{anaumov@hse.ru}

\author[b,c]{Nikita Puchkin}
\ead{npuchkin@hse.ru}

\author[b]{Sergey Samsonov}
\ead{svsamsonov@hse.ru}

\address[a]{Duisburg-Essen University, Germany}
\address[b]{HSE University, Russia}
\address[c]{Institute for Information Transmission Problems RAS, Russia}

\begin{abstract}
This paper investigates the approximation properties of deep neural networks with  piecewise-polynomial activation functions.  We derive the required depth, width, and sparsity of a deep neural network to approximate any H\"{o}lder smooth function up to a given approximation error in H\"{o}lder norms in such a way that all weights of this neural network are bounded by $1$. The latter feature is essential to control generalization errors in many statistical and machine learning applications.
\end{abstract}


\end{frontmatter}

%\linenumbers

\section{Introduction}
\label{sec:intro}
Neural networks have recently gained much attention due to their impressive performance in many complicated practical tasks, including image processing \cite{lecun15}, generative modelling \cite{goodfellow14}, reinforcement learning \cite{mnih15}, numerical solution of PDEs, e.g., \cite{han18, geist21}, and optimal control \cite{chen19, onken22}. This makes them extremely useful in design of self-driving vehicles \cite{li18} and robot control systems, e.g., \cite{cembrano94, bozek17, gonzalez22}. One of the reasons for such a success of neural networks is their expressiveness, that is, the ability to approximate functions with any desired accuracy. The question of expressiveness of neural networks has a long history and goes back to the papers \cite{funahashi89, cybenko89, hornik91}.
In particular, in \cite{cybenko89}, the author showed that one hidden layer is enough to approximate any continuous function $f$ with any prescribed accuracy $\eps > 0$. However, further analysis revealed the fact that deep neural networks may require much fewer parameters than the shallow ones to approximate $f$ with the same accuracy. Many efforts were put in recent years to understand the fidelity of deep neural networks. 
In a pioneering work \cite{yarotsky17}, the author showed that for any target function $f$ from the Sobolev space $\W^{n, \infty}([0, 1]^d)$ there is a neural network with $O(\eps^{-d/n})$ parameters and ReLU activation function, that approximates $f$ within the accuracy $\eps$ with respect to the $L_\infty$-norm on the unit cube $[0, 1]^d$. Further works in this direction considered various smoothness classes of the target functions \cite{li20b,guhring21,lu21,shen21}, neural networks with diverse activations \cite{guhring21,deryck21,jiao21,langer21}, domains of more complicated shape \cite{shen20}, and measured the approximation errors with respect to different norms \cite{yarotsky17,guhring21,deryck21,schmidt-hieber20}.  Several authors also considered the expressiveness of neural networks with different architectures. This includes wide neural networks of logarithmic \cite{yarotsky17,guhring21,schmidt-hieber20} or even constant depth \cite{li20a,li20b,deryck21,shen21}, or deep and narrow networks \cite{hanin19,kidger20,park21}. Most of the existing results on the expressiveness of neural networks measure the quality of approximation with respect to either the $L_\infty$- or $L_p$-norm, $p \geq 1$. Much fewer papers study the approximation of derivatives of smooth functions. These rare exceptions include \cite{guhring20, guhring21, deryck21}.
\par
In the present paper, we focus on feed-forward neural networks with  piecewise-polynomial activation functions of the form $\sigma^{\mathsf{ReQU}}(x) = (x \vee 0)^2$.
%The ReQU activation function is a particular case of the RePU activations $\sigma^{\mathsf{RePU}}(x) = (x \vee 0)^p$, $p > 1$.
Neural networks with such activations are known to successfully approximate smooth functions from the Sobolev and Besov spaces with respect to the $L_\infty$- and $L_p$-norms (see, for instance, \cite{klusowski18, li20a, li20b, ali21, abdeljawad22, chen22, gribonval22, siegel22}). We continue this line of research and study the ability of such neural networks to approximate not only smooth functions themselves but also their derivatives. We derive the non-asymptotic upper bounds on the H\"{o}lder norm between the target function and its approximation from a class of sparsely connected neural networks with ReQU activations. In particular, we show that, for any $f$ from a H\"{o}lder ball $\H^\beta([0,1]^d, H)$, $H > 0$, $\beta > 2$, (see \Cref{sec:notations} for the definition) and any $\eps > 0,$ there exists a neural network with ReQU activation functions, that uniformly approximates the target function in the norms of the H\"{o}lder spaces $\H^{\ell}([0,1]^d)$ for all $\ell \in \{0,\ldots,\lfloor \beta \rfloor\}$. Here and further in the paper, $\lfloor \beta \rfloor$ stands for the largest integer which is \emph{strictly smaller} than $\beta$. A simplified statement of our main result is given below. 

\begin{Thm}[simplified version of \Cref{prop:requ_approx}]
\label{prop:main}
	Fix  $\beta > 2$ and $p, d \in \mathbb N$.
	Then, for any $H > 0$, for any $f : [0, 1]^d \rightarrow \R^p$, $f \in \H^\beta([0,1]^d, H)$ and any integer $K \geq 2$, there exists a neural network $h_f : [0, 1]^d \rightarrow \R^p$ with ReQU activation functions  such that it has $O(\log d + \lbeta + \log \log H)$ layers,  at most $O(p \vee d(K+\lbeta)^{d})$ neurons in each layer and $O(p (d\beta + d^2 + \log\log H) (K+\lbeta)^{d})$ non-zero weights taking their values in $[-1, 1]$. Moreover, it holds that
    \[
        \left\|f - h_f\right\|_{\H^\ell([0, 1]^d)} \leq \frac{C^{\beta d} H \beta^\ell}{K^{\beta-\ell}}
	    \quad
	    \text{for all $\ell \in \{0, \ldots, \lfloor\beta\rfloor\}$,}
    \]
where $C$ is a universal constant. 
\end{Thm}

We provide explicit expressions for the hidden constants in \Cref{prop:requ_approx}.
%The quality of approximation of neural networks in terms of norms involving derivatives was %already studied in the previous literature, see \cite{guhring20, guhring21, deryck21}. Let us %elaborate on the contributions of the present paper and relate them to the existing literature.
The main contributions of our work can be summarized as follows.
\begin{itemize}
    \item Given a smooth target function $f \in \H^\beta([0,1]^d, H)$, we construct a neural network, that \emph{simultaneously} approximates all the derivatives of $f$ up to order $\lfloor\beta\rfloor$ with optimal dependence of the precision on the number of non-zero weights. That is, if we denote the number of non-zero weights in the network by $N$, then it holds that $\left\|f - h_f\right\|_{\H^\ell([0, 1]^d)} = O(N^{-(\beta-\ell)/d})$ simultaneously for all $\ell \in \{0, \ldots, \lfloor\beta\rfloor\}.$ %In contrast, \cite{guhring20, guhring21, deryck21} need to use neural networks of different architectures for different $\ell \in \{0, \dots, \lfloor\beta\rfloor\}$.
    \item The constructed neural network has almost the same smoothness as the target function itself while approximating it with the optimal accuracy. This property turns out to be very useful in many applications including the approximation of PDEs and density transformations where we need to use derivatives of the approximation. %No similar bounds on the derivatives of the approximating neural network can be derived from the results in \cite{guhring20, guhring21, deryck21}.
    \item The weights of the approximating neural network are bounded in absolute values by $1$. The latter property plays a crucial role in deriving bounds on the generalization error of empirical risk minimizers in terms of the covering number of the corresponding parametric class of neural networks. Note that the upper bounds on the weights provided in \cite{guhring20, guhring21, deryck21} blow up as the approximation error decreases.
\end{itemize}
The rest of the paper is organized as follows. In \Cref{sec:notations}, we introduce necessary definitions and notations. \Cref{sec:main_result} contains the statement of our main result, \Cref{prop:requ_approx}, followed by a detailed comparison with the existing literature. We then present numerical experiments in \Cref{sec:numerical}. The proofs are collected in \Cref{sec:proofs}. Some auxiliary facts are deferred to \ref{app:splines}.
 
\section{Preliminaries and notations}
\label{sec:notations}
\paragraph*{Norms}
For a matrix $A$ and a vector $v$, we denote by $\|A\|_\infty$ and $\|v\|_\infty$  the maximal 
absolute value of entries of $A$ and $v$, respectively.
$\|A\|_0$ and $\|v\|_0$ shall stand for the number of non-zero entries of $A$ and $v$, respectively. Finally, the Frobenius norm and operator norm of $A$ are denoted by $\|A\|_F$ and 
$\|A\|$, respectively, and the Euclidean norm of $v$ is denoted by $\|v\|$. For a function $f : \Omega\to \mathbb{R}^d$, we set 
\begin{align*}
	&
	\|f\|_{L_\infty(\Omega)} = \esssup_{x\in \Omega} \|f(x)\|,
	\\&
	\|f\|_{L_p(\Omega)} =  \left\{\int_{\Omega}\|f(x)\|^p\,\rmd x\right\}^{1/p},
	\quad
	p \geq 1.
\end{align*}
If the domain $\Omega$ is clear from the context, we simply write $L_\infty$ or $L_p$, instead of $L_\infty(\Omega)$ or $L_p(\Omega)$, respectively.

\paragraph*{Smoothness classes}
Let $\Omega \subseteq \rset^{d}$, $f: \Omega \mapsto \rset^{m}$. For a multi-index $\bgamma = (\gamma_1,\dots,\gamma_d) \in 
\mathbb{N}_0^d$, we write $|\bgamma| = \sum_{i=1}^d \gamma_i$ and define the corresponding partial differential operator $D^{\bgamma}$ as
\[
D^{\bgamma}f_i = \frac{\partial^{|\bgamma|} f_i}{\partial x_1^{\gamma_1} \cdots \partial x_d^{\gamma_d}}, \quad i \in \{1,\dots, m\}\,, \text{ and } \|D^{\bgamma} f\|_{L_\infty(\Omega)} = \max\limits_{1 \leq i \leq m} \|D^{\bgamma} f_i\|_{L_\infty(\Omega)}\,.
\]
For $s \in \mathbb N$, the function space $C^{s}(\Omega)$ consists of those 
functions over the domain $\Omega$ which have bounded and continuous derivatives up to order $s$ in $\Omega$. Formally,
\[
	\textstyle{C^s(\Omega) = \big\{ f: \Omega \to \R^m: \quad \|f\|_{C^s} := \max\limits_{ 
	|\bgamma| \leq s} \|D^{\bgamma} f\|_{L_\infty(\Omega)} < \infty \big\},} 
\]
For a function $f:\Omega\to\R^m$ and any positive number $0 < \delta \leq 1$, the
\emph{H\"older constant} of order $\delta$ is given by
\begin{equation}\label{beta}
[f]_{\delta}:=\max_{i \in \{1,\ldots,m\}}\sup_{x \not = y\in\Omega}\frac{|f_i(x)-f_i(y)|}{\min\{1, 
\|x-y\|\}^{\delta}}\;.
\end{equation} 
Now, for any $\beta >0$, we can define the \emph{H\"older ball} 
$\H^{\beta}(\Omega,H)$. If we let $s = \lbeta$ be the largest integer \emph{strictly less} than $\beta$, $\H^{\beta}(\Omega, H)$ contains all functions in $C^s(\Omega)$ with $\delta$-H\"older-continuous, $\delta = \beta - s > 0$, partial derivatives of order $s$. Formally,
\[
\textstyle{
	\H^{\beta}(\Omega,H) = \big\{ f \in C^s(\Omega): \quad \|f\|_{\H^{\beta}} := \max \{ 
	\|f\|_{C^s}, \ \max\limits_{|\bgamma| = s} [D^{\bgamma}f]_{\delta} \} \leq H \big\}.}
\]
%Note that if \(f\in \H^{1+\beta}(\Omega,H)\) for some \(\beta>0,\) then it holds
%for  \(i=1,\ldots,m,\) $j = 1, \ldots, d$,
%\[
%\left|\frac{\partial f_i(x)}{\partial x_j} - \frac{\partial f_i(y)}{\partial x_j} \right| 
%\leq \|f\|_{\H^{1+\beta}} \cdot \|x-y\|^{1\wedge \beta}\leq H\cdot \|x-y\|^{1\wedge \beta}, 
%\quad x,y\in\Omega
%\] 
%for any \(x,y\in \Omega,\) since $\|f\|_{\H^{\beta_1}}\leq\|f\|_{\H^{\beta_2}} $ for any 
%$\beta_2\geq \beta_1$. 
We also write $f\in \H^{\beta}(\Omega)$ if  $f \in \H^{\beta}(\Omega,H)$ for some $H<\infty.$

\paragraph*{Neural networks}
Fix an activation function 
$\sigma: \R \rightarrow \R$. For a  vector $v= (v_1,\dots,v_p) \in \R^p$,
we define the shifted activation function $\sigma_v: \R^p \rightarrow \R^p$ as 
\begin{equation*}
	\sigma_{v}(x) = \bigl(\sigma(x_1-v_1),\dots,\sigma(x_p-v_p)\bigr), \quad x = 
	(x_1,\dots,x_p) 
	\in 
	\R^p.
\end{equation*}
Given a positive integer $L$ and a vector $\A = (p_0, p_1, \dots, p_{L+1}) \in \mathbb N^{L+2}$,  a neural network of depth $L+1$ (with $L$ hidden layers) and architecture \(\A\) is  a function of the form
\begin{equation}
	\label{eq:nn}
	f: \R^{p_0} \rightarrow \R^{p_{L+1}}\,, \quad 
	f(x) = W_L \circ \sigma_{v_{L}} \circ W_{L-1} \circ \sigma_{v_{L-1}} \circ \dots \circ W_1\circ\sigma_{v_1} \circ 
	W_0 \circ x,
\end{equation}
where $W_i \in \R^{p_{i+1} \times p_i}$ are weight matrices and $v_i \in \R^{p_i}$ are shift (bias) vectors. The maximum number of neurons in one layer $\|\A\|_\infty$ is called the width of the neural network. Next, we introduce a subclass $\NN(L, \A, s)$ of neural networks of depth $L+1$ with architecture $\A$ and at most $s$ non-zero weights. That is, $\NN(L, \A, s)$ consist of functions of the form \eqref{eq:nn}, such that
\[
\begin{cases}
\|W_0\|_\infty \vee \max\limits_{1 \leq \ell \leq L} \left\{ \|W_\ell\|_\infty \vee \|v_\ell\|_\infty \right\} \leq 1\,, \\
\|W_0\|_0 + \sum\limits_{\ell = 1}^L \left(\|W_\ell\|_0 + \|v_\ell\|_0 \right) \leq s\,.
\end{cases}
\]
We also use the notation $\NN(L, \A)$, standing for $\NN(L, \A, \infty)$.
%The most common choice for the activation function is the rectified linear unit (ReLU):
%\begin{equation*}
%	\sigma^{\mathsf{ReLU}}(x) = x \vee 0.
%\end{equation*}
%Unfortunately, it is not suitable for our purposes.
%The reason is that we want to use neural networks as e.g. generators.
Throughout the paper, we use the ReQU (rectified quadratic unit) activation function, defined as
\begin{equation*}
	\sigma(x) = (x \vee 0)^2.
\end{equation*}

\paragraph*{Concatenation and parallelization of neural networks}

During the construction of approximating network in \Cref{prop:requ_approx}, we use the operations of concatenation (consecutive connection) and parallel connection of neural networks. Given neural networks $g$ and $h$ of architectures $(p_0, p_1, \dots, p_L, p_{out})$ and $(p_{in}, p_{L+1}, p_{L+2}, \dots, p_{L+M+1})$,  respectively, such that $p_{in} = p_{out}$, the concatenation of $g$ and $h$ is their composition $h \circ g$, that is, a neural network of the architecture $(p_0, p_1, \dots, p_L, p_{L+1}, p_{L+2}, \dots, p_{L+M+1})$. The parallel connection of neural networks is defined as follows. Let
\[
    f(x)  = W_L \circ \sigma_{v_{L}} \circ W_{L-1} \circ \sigma_{v_{L-1}} \circ \dots \circ W_1\circ\sigma_{v_1} \circ 
	W_0 \circ x
\]
and
\[
    \widetilde f(x)  = \widetilde W_L \circ \sigma_{\widetilde v_{L}} \circ \widetilde W_{L-1} \circ \sigma_{\widetilde v_{L-1}} \circ \dots \circ \widetilde W_1\circ\sigma_{\widetilde v_1} \circ 
	\widetilde W_0 \circ x
\]
be two neural networks of the same depth with architectures $(p_0, p_1, \dots, p_L, p_{L+1})$ and $(p_0, \widetilde p_1, \widetilde p_2 \dots, \widetilde p_L, \widetilde p_{L+1})$, respectively. Define
\[
    \overline W_0 =
    \begin{pmatrix}
        W_0 \\ \widetilde W_0.
    \end{pmatrix}\,,
    \quad
    \overline W_\ell =
    \begin{pmatrix}
        W_\ell & O \\ O & \widetilde W_\ell
    \end{pmatrix}\,,
    \quad \text{and} \quad
    \overline v_\ell =
    \begin{pmatrix}
        v_\ell \\ \widetilde v_\ell
    \end{pmatrix}
    \quad \text{for all $\ell \in \{1, \dots, L\}$}.
\]
Then the parallel connection of $f$ and $\widetilde f$ is a neural network $\overline f$ of architecture $(p_0, p_1 + \widetilde p_1, p_2 + \widetilde p_2, \dots, p_L + \widetilde p_L, p_{L + 1} + \widetilde p_{L + 1})$, given by 
\[
    \overline f(x)  = \overline W_L \circ \sigma_{\overline v_{L}} \circ \overline W_{L-1} \circ \sigma_{\overline v_{L-1}} \circ \dots \circ \overline W_1\circ\sigma_{\overline v_1} \circ 
	\overline W_0 \circ x.
\]
Note that the number of non-zero weights of $\overline f$ is equal to the sum of the ones in $f$ and $\widetilde f$.

%From now on, we always use ReQU activation functions and write $\sigma(x)$ instead of $\sigma^{\mathsf{ReQU}}(x)$ for short. 
%In what follows, we consider sparse neural networks assuming that only a few weights are not equal to zero. For this purpose, we introduce a class of neural networks of depth $L+1$ with architecture $\A$ and with at most $s$ non-zero weights:
%\[
%	\NN(L, \A, s) =
%	\left\{ f \in \NN(L, \A) : \|W_0\|_0 + \sum\limits_{\ell = 1}^L \left(\|W_\ell\|_0 + 
%	\|v_\ell\|_0 \right) \leq s\right\}.
%\]

\section{Main results}
\label{sec:main_result}
\subsection{Approximation of functions from H\"{o}lder classes}

Our main result states that any function from $\H^\beta([0,1]^d, H)$, $H > 0$, $\beta > 2$, can be approximated by a feed-forward deep neural network with ReQU activation functions in  $\H^\ell([0,1]^d),$ $\ell \in \{0, \dots, \lfloor\beta\rfloor\}.$

\begin{Thm}
\label{prop:requ_approx}
	Let $\beta > 2$ and let $p, d \in \mathbb N$.
	Then, for any $H > 0$, for any $f : [0, 1]^d \rightarrow \R^p$, $f \in \H^\beta([0,1]^d, H)$ and any integer $K \geq 2$, there exists a neural network $h_f : [0, 1]^d \rightarrow \R^p$ of the width
	\[
	    \bigl(4d (K + \lbeta)^d \bigr) \vee 12\,\left((K + 2\lbeta) + 1\right) \vee p
    \]
    with
    \[
        6 + 2(\lbeta-2) + \lceil \log_{2}{d} \rceil + 2\left( \left\lceil \log_2 (2d\lbeta + d) \vee \log_2 \log_2 H \right\rceil \vee 1 \right)
    \]
    hidden layers and at most $p (K + \lbeta)^{d} C(\beta, d, H)$ non-zero weights taking their values in $[-1, 1]$, such that, for any $\ell \in \{0, \dots, \lfloor\beta\rfloor\}$,
	\begin{equation}
		\label{eq:approx_H_l_norm}
		\left\|f - h_f\right\|_{\H^\ell([0, 1]^d)}
		\leq  \frac{ (\sqrt{2}e d)^{\beta} H}{K^{\beta-\ell}} + \frac{9^{d(\lbeta - 1)} (2\lbeta + 1)^{2d + \ell} (\sqrt{2}ed)^\beta H}{K^{\beta - \ell}}.
	\end{equation}
	The above constant $C(\beta, d, H)$ is given by
	\[
	    C(\beta, d, H)
	    = \left( 60 \bigl( \left\lceil \log_2 (2d\lbeta + d) \vee \log_2 \log_2 H \right\rceil \vee 1 \bigr) + 38 \right) 
	    + 20 d^2 + 144 d \lbeta + 8 d.
	\]
\end{Thm}
An illustrative comparison of the result of \Cref{prop:requ_approx} with the literature is provided in Table \ref{table:comparison} below.
\Cref{prop:requ_approx} improves over the results of \cite[Theorem 3.3]{li20a} and \cite[Theorem 7]{li20b} as far as the approximation properties of ReQU neural networks in terms of the $L_\infty$-norm are concerned. Namely, \Cref{prop:requ_approx} claims that, for any $f \in \H^\beta([0, 1]^d, H)$, $\beta > 2$, and any $\eps > 0$, there is a ReQU neural network of width $O(\eps^{-d/\beta})$ and depth $O(1)$ with $O(\eps^{-d/\beta})$ non-zero weights, taking their values in $[-1, 1]$, such that it approximates $f$ within the accuracy $\eps$ with respect to the $L_\infty$-norm on $[0, 1]^d$. In \cite{li20a, li20b}, the authors considered a target function from a  general weighted Sobolev class but measure the quality of approximation in terms of a  weighted $L_2$-norm. Nevertheless, the width and the number of non-zero weights of the constructed neural networks in \Cref{prop:requ_approx}, \cite[Theorem 3.3]{li20a}, and \cite[Theorem 7]{li20b} coincide. The difference is that, while in \cite{li20a, li20b}, the depth of the neural networks is of order $O(\log(1/\eps)),$  we need only $O(1)$ layers. More importantly, the authors of \cite{li20a, li20b} do not provide any guarantees on the absolute values of the weights of the approximating neural networks. At the same time, all the weights  of our neural network take their values in $[-1, 1]$. We will explain the importance of the latter property a bit later in this section. It is also worth mentioning that the same number of non-zero weights $O(\eps^{-d/\beta})$ is needed to approximate $f$ within the tolerance $\eps$ (with respect to the $L_\infty$-norm) via neural networks with other activation functions, such as ReLU \cite{yarotsky20}, sigmoid \cite{langer21, guhring21}, and $\arctan$ \cite{guhring21}.

The approximation properties of deep neural networks with respect to norms involving derivatives are much less studied. In the rest of this section, we elaborate on the comparison with the state-of-the-art results %\cite{guhring21, deryck21}
in this direction.
The fact that shallow neural networks can simultaneously approximate a smooth function and its derivatives is known for a long time from the paper \cite{xu05}, where the authors derived an upper bound on the approximation accuracy in terms of the modulus of smoothness. To be more precise, they showed that if a function $f$ is continuous on a compact set $\mathcal K \subset \R^d$ and its derivatives up to an order $s \in \mathbb Z_+$ are in $L_p(\mathcal K)$, then there is a neural network $g_f$ with
\[
    (K+1) \sum\limits_{j=0}^K \binom{j + d - 1}{d - 1} 
    = O(K^{d+1})
\]
hidden units, such that
\[
    \left\| D^\bgamma g_f - D^\bgamma f \right\|_{L_p(\mathcal K)}
    \lesssim \omega \left(D^\bgamma f, \frac1{\sqrt{K}} \right)_p + \frac{\|D^\bgamma f\|_{L_p(\mathcal K)}}{K}
    \quad \text{for all $\bgamma \in \Z_+^d$ such that $|\bgamma| \leq s$.}
\]
Here
\[
    \omega(\phi, \delta)_p = \sup\limits_{0 < t \leq \delta} \|\phi(\cdot + t) - \phi(\cdot)\|_{L_p(\mathcal K)},
    \quad \delta > 0,
\]
is the modulus of smoothness of a function $\phi$.
Moreover, if, in addition, the derivatives of $f$ of order $s$ are $\alpha$-H\"older, $\alpha \in (0, 1]$, then it holds that
\[
    \|f - g_f\|_{\H^s(\mathcal K)} \lesssim K^{-\alpha/2}.
\]
Taking into account that $\omega(\phi, \delta)_p \lesssim \delta$ for a Lipschitz function $\phi$, we conclude that one has to use a shallow neural network with at least $\Omega(\eps^{-2(d+1)})$ hidden units to approximate a function of interest or its derivatives within the accuracy $\eps$ with respect to the $L_p(\mathcal K)$-norm even if $f$ is sufficiently smooth. This bound becomes prohibitive in the case of large dimension. The situation is much better in the case of deep neural networks.
To our knowledge, G\"{u}hring and Raslan \cite[Proposition 4.8]{guhring21} were the first to prove that, for any $f \in \H^\beta([0, 1]^d, H)$, $\eps > 0$, and $\ell \in \{0, \dots, \lbeta\}$, there is a ReQU neural network with $O(\eps^{-d/(\beta - \ell)})$ non-zero weights, which approximates $f$ within the accuracy $\eps$ with respect to the $\H^\ell$-norm on $[0, 1]^d$.
%Shallow ReQU networks have much worse guarantees \cite[Theorem 3]{siegel22}. 
A drawback of the result in \cite{guhring21} is that the architecture of the suggested neural network heavily depends on $\ell$. Hence, it is hard to control higher-order derivatives of the approximating neural network itself, which is of interest in such applications as numerical solutions of PDEs and density transformations. This question was addressed in \cite{deryck21, hon21}, where the authors considered the problem of \emph{simultaneous} approximation of a target function with respect to the H\"{o}lder norms of different orders. In \cite[Theorem 5.1]{deryck21}, the authors showed that, for any  $f \in \H^s([0, 1]^d, H)$, $s \in \mathbb N$, and any sufficiently large integer $K$, there is a three-layer neural network $g_f$ of width $O(K^d)$ with $\tanh$ activations, such that
\[
    \|f - g_f\|_{\H^\ell([0, 1]^d)} = O\left( \frac{(\log K)^\ell}{K^{s - \ell}} \right)
\]
simultaneously for all $\ell \in \{0, 1, \dots, s-1\}$. Note that \Cref{prop:requ_approx} yields a sharper bound, removing the odd logarithmic factors. Regarding the results of the recent work \cite{hon21}, we found some critical flaws in the proofs of the main results. We explain our concerns in \Cref{rem:hon} below as the main results of \cite{hon21} are close to ours. \Cref{prop:requ_approx} also improves over \cite{guhring21, deryck21, hon21} in another aspect. Namely, \Cref{prop:requ_approx} guarantees that the weights of the neural network take their values in $[-1, 1]$, while in \cite{guhring21, deryck21, hon21} they may grow polynomially as the approximation error decreases. This boundeness property is extremely important when studying the generalization ability of the neural networks in various statistical problems since the metric entropy of the corresponding parametric class of neural networks involves a uniform upper bound on the weights (see, for instance, \cite[Theorem 2 and Lemma 5]{schmidt-hieber20}).  Besides that, a polynomial upper bound on the absolute values of the weights becomes prohibitive when approximating analytic functions (see \Cref{rem:deryck} in the next section).

\begin{Rem}
    \label{rem:hon}
    The proofs of Theorem 3.1 and Theorem 3.8 in \cite{hon21} have  critical flaws. %We think that it should be clarified to a reader, since the main results of \cite{hon21} are extremely close to ours. 
    In particular, on page 18, the authors mistakenly bound the Sobolev norm $\W^{1, \infty}$ of the composite function $\phi(x) = \widetilde\phi(\phi_\alpha(\psi(x)) / \alpha!, P_\alpha(h))$ by the $\W^{1, \infty}$-norm of $\widetilde\phi$. A similar error appears on page 24. In this way, the authors obtain that the $\W^{1, \infty}$-norm of $\phi(x)$ is bounded by $432 s^d$, where $s$ is the smoothness of the target function. In fact, the latter norm should scale as $1/\delta$, where $\delta > 0$ is a small auxiliary parameter describing the width of the boundary strips. This flaw completely ruins the proofs of the main results in \cite{hon21}, Theorem 1.1 (see the 9th line of the proof on p.8) and  Theorem 1.4 (the 9th line of the proof on p.25).
\end{Rem}

\begin{table}[t]
    \caption{comparison of the state-of-the-art results on approximation of a function $f \in \H^\beta([0,1]^d, H), \beta > 2, H > 0$, within the accuracy $\eps$ via neural networks. The papers marked with $*$ consider the case of integer $\beta$ only.}
    \noindent\centering
    \begin{tabular}{llllll}
    \toprule
    Paper & Norm & Depth & \makecell[lc]{Non-zero \\ weights} & \makecell[lc]{Simultaneous \\ approximation} & \makecell[lc]{Weights \\ in $[-1, 1]$} \\
    \midrule
    \cite{li20b}* & $L_2$ & $O(\log(1/\eps))$ & $O(\eps^{-d/\beta})$ & N/A & \no \\
    \hline
    \cite{guhring21}* & $\H^\ell$ & $O(\log(1/\eps))$ & $O(\eps^{-d/(\beta - \ell)})$ & \no & \no \\
    \hline
    \cite{deryck21} & $\H^\ell$ & $3$ & $O(\eps^{-d/(\beta - \ell)} (\log(1/\eps))^\ell)$ & \yes & \no \\
    \hline
    \cite{langer21} & $L_\infty$ & $O(\log(1/\eps))$ & $O(\eps^{-d/\beta})$ & N/A & \no \\
    \hline
    \cite{schmidt-hieber20} & $L_\infty$ & $O(\log(1/\eps))$ & $O(\eps^{-d/\beta})$ & N/A & \yes \\
    \hline
    \cite{li20a}* & $L_2$ & $O(\log(1/\eps))$ & $O(\eps^{-d/\beta})$ & N/A & \no \\
    \hline
    \cite{xu05}* & $\H^\ell$ & 2 & $O(\eps^{-2(d+1) /(1 \wedge (\beta - \ell))})$ & \yes & \no \\
    \hline
    Ours & $\H^\ell$ & $O(1)$ & $O(\eps^{-d/(\beta - \ell)})$ & \yes & \yes \\
    \bottomrule
    \end{tabular}
    \label{table:comparison}
\end{table}

\subsection{Approximation of analytic functions}

The bound of \Cref{prop:requ_approx} can be transformed to exponential (in $K$) rates of approximation for analytic functions. In the rest of this section, we consider $(Q, R)$-analytic functions defined below.
\begin{Def}
    A function $f \in C^{\infty}(\rset^{d})$ is called $(Q,R)$-analytic with $Q, R > 0$, if it satisfies the inequality
    \begin{equation}
        \label{eq:holder_norm_analyt}
        \left\|f\right\|_{\H^s([0, 1]^d)} \leq Q R^{-s} s! \quad \text{for all $s \in \mathbb N_0$.}
    \end{equation}
\end{Def}
Similar concepts were also considered in \cite{candes07, candes09, deryck21}. Applying \Cref{prop:requ_approx} to $(Q, R)$-analytic functions, we get the following corollary.

\begin{Co}
\label{co:analyt_exp_bound}
Let $f$ be a $(Q,R)$-analytic function and let $\ell \in \mathbb N_0$. Then, for any integer $K > 2ed \, 9^d / R$, there exists a neural network $h_{f}: [0,1]^{d} \mapsto \rset^p$ of width $O((K + \ell)^d \vee p)$ with $O(K + \ell)$ hidden layers and at most $O(p(K + \ell)^d \log K)$ non-zero weights, taking values in $[-1, 1]$, such that 
    \begin{equation}
        \label{eq:analyt_bound}
        \left\|f - h_{f}\right\|_{\H^\ell([0, 1]^d)} \leq 
        \frac{Q e (\sqrt{2} \, e d \, 9^d)^\ell (2s - 1)^{2d + 2\ell}}{R^\ell} \exp\left\{ -\left\lfloor \frac{KR}{2ed \, 9^{d}} \right\rfloor \right\}.
    \end{equation}
\end{Co}

\begin{Rem}
\label{rem:deryck}
In \cite[Corollary 5.6]{deryck21}, the authors claim that $(Q,R)$-analytic  functions can be approximated with exponential (in the number of non-zero weights $N$) rates with respect to the Sobolev norm $\W^{k, \infty}([0,1]^d)$, where $k$ is a fixed integer (see eq.(112) on p.743). However, a closer look at Corollary 5.6 reveals the fact that the right-hand side of eq.(112) includes a constant $c_{d,k,\alpha, f}$ with $\alpha$ depending on $N$ too, as follows from eq.(54). Thus, the dependence of the final bound on $N$  in Corollary 5.6 remains unclear. Besides that, in contrast to \cite[Theorem 5.2]{deryck21}, the authors do not specify the upper bound on the absolute values of the weights of the constructed neural network in Corollary 5.6. A thorough inspection of the proof shows that the weights can be as large as $O(N^{N^2})$. 
\end{Rem}

\section{Numerical experiments}
\label{sec:numerical}
In this section, we provide numerical experiments to illustrate the approximation properties of neural networks with ReQU activations. We considered a scalar function $f(x) = \sin(x_1^2 x_2)$ of two variables and approximated it on the unit square $[0, 1]^2$ via neural networks with two types of activations: ReLU and ReQU. All the neural networks were fully connected, and all their hidden layers had a width $16$. The first layer had a width $2$. The depth of neural networks took its values in $\{1, 2, 3, 4, 5\}$. In the training phase, we sampled $N = 10000$ points $X_1, \dots, X_N$ independently from the uniform distribution on $[0, 1]^2$ and tuned the weights of neural networks by minimizing the mean empirical squared error
\[
    \mathrm{ERR}(h) = \frac1N \sum\limits_{i=1}^N (h(X_i) - f(X_i))^2.
\]
After that, we computed the approximation errors of $f$ and of its gradient on the two-dimensional grid $G = \{0, 1/M, \dots, (M - 1) / M, 1\}^2$ with $M = 500$:
\[
    \frac1{M^2} \sum\limits_{x \in G} (\widehat h(x) - f(x))^2
    \quad \text{and} \quad
    \frac1{M^2} \sum\limits_{x \in G} \left\| \nabla\widehat h(x) - \nabla f(x) \right\|^2,    
\]
where $\widehat h$ denotes the neural network with the weights tuned on the training phase. We repeated the experiment $10$ times and computed average approximation errors. The results are displayed in \Cref{fig:err}. The quality of approximation by neural networks with ReQU activations turns out to be better than by the ones with ReLU. Note that in such a scenario the stochastic error becomes small, and the quality of learning is mostly determined by the approximation error. This claim is supported by the fact that the error values were similar in different experiments.

\begin{figure}[ht]
    \noindent\centering
    \includegraphics[width=\linewidth]{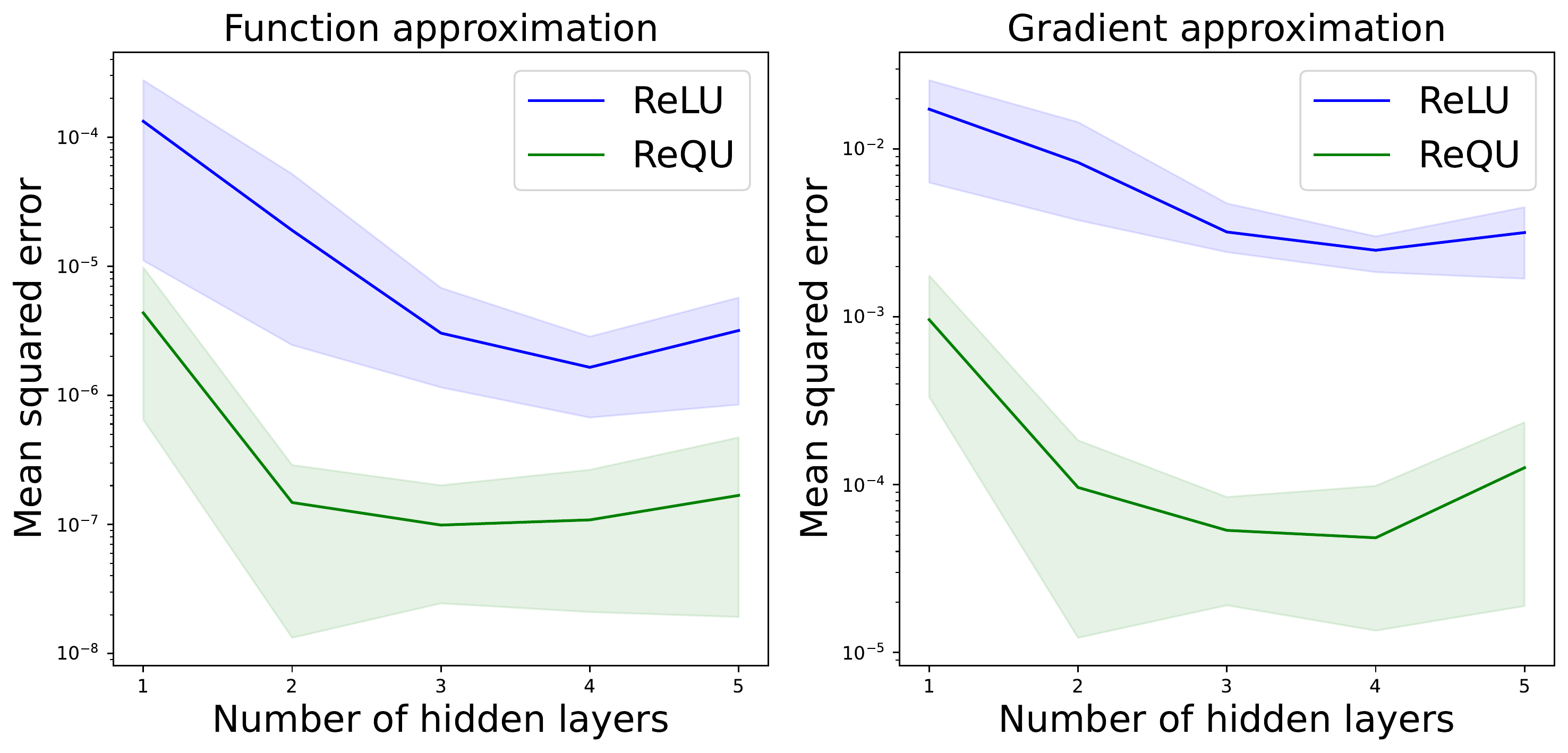}
    \caption{The mean squared error of approximation of the target function $f = \sin(x_1^2 x_2)$ (left) and its gradient (right) by neural networks with different activations.}
    \label{fig:err}
\end{figure}

\section{Proofs}
\label{sec:proofs}
\subsection{Proof of \Cref{prop:requ_approx}}

\noindent{\bf Step 1.}\quad Let $f = (f_1, \dots, f_p)$. Consider a vector $a = (a_1, \dots, a_{2\lbeta + K + 1})$, such that
\begin{align}
    \label{eq:knots}
    &\notag
    a_1 = \ldots = a_{\lbeta + 1} = 0\,;
    \\&
    a_{\lbeta + 1 + j} = j/K, \quad 1 \leq j \leq K - 1\,;
    \\&\notag
    a_{\lbeta + K + 1} = \ldots = a_{2\lbeta + K + 1} = 1\,.
\end{align}
By \Cref{thm:spline_approx}, there exist tensor-product splines $S_f^{\lbeta, K} = (S_{f, 1}^{\lbeta, K}, \dots, S_{f, p}^{\lbeta, K})$ of order $\lbeta \geq 2$ associated with knots at $\big\{(a_{j_1}, a_{j_2}, \dots, a_{j_d}) : j_1, \dots, j_d \in \{1, \dots, 2\lbeta + K + 1\}\big\}$ such that
\begin{align*}
    \left\|f - S_f^{\lbeta, K} \right\|_{\H^\ell([0, 1]^d)} 
    &
    \leq \max_{m \in \{1, \dots, p\}} \left\|f_m - S_{f, m}^{\lbeta, K} \right\|_{\H^\ell([0, 1]^d)}
    \\&
    \leq \frac{ (\sqrt{2}e d)^{\beta} H}{K^{\beta-\ell}} + \frac{9^{d(\lbeta - 1)} (2\lbeta + 1)^{2d + \ell} (\sqrt{2}ed)^\beta H}{K^{\beta - \ell}}.
\end{align*}
Our goal is to show that $S_{f}^{\lbeta, K}$ can be represented by a neural network $h_{f}$ with ReQU activation functions. For this purpose, for each $m \in \{1, \dots, p\}$, we use an expansion of $S_{f, m}^{\lbeta, K}$ with respect to the basis
\[
\big\{N_{j_1}^{\lbeta, K}(x_1) \cdot N_{j_2}^{\lbeta, K}(x_2) \cdot \ldots \cdot N_{j_d}^{\lbeta, K}(x_d) : j_1, \dots,j_d \in \{1, \dots,\lbeta+K\} \big\},
\]
where $N^{\lbeta, K}_1, \dots, N^{\lbeta, K}_{\lbeta+K}$ are normalized $B$-splines defined in \eqref{eq:norm_b-spline}. There exist coefficients $\{w^{(f)}_{m, j_1, j_2, \dots, j_d} : 1 \leq m \leq p, 1 \leq j_1, \dots, j_d \leq \lbeta + K\}$ such that, for any $m \in \{1, \dots, p\}$,
\begin{equation}
\label{eq:spline_expansion}
\begin{split}
S_{f, m}^{\lbeta, K} 
&= \sum\limits_{j_1, \dots, j_d = 1}^{\lbeta+K} w^{(f)}_{m,j_1, \dots, j_d} \prod\limits_{\ell=1}^d N_{j_{\ell}}^{\lbeta,K}(x_{\ell}) \\
&= \sum\limits_{j_1, \dots, j_d = 1}^{\lbeta+K} w^{(f)}_{m,j_1, \dots, j_d} \prod\limits_{\ell = 1}^d (a_{j_\ell+\lbeta+1} - a_{j_\ell}) B_{j_\ell}^{\lbeta,K}(x_\ell)\,,
\end{split}
\end{equation}
where $B_{j_\ell}^{\lbeta,K}(x_\ell)$ are (unnormalized) $B$-splines defined in \eqref{eq:b-spline_recursion}. Hence, in order to represent $S_f^{\lbeta, K} = (S_{f, 1}^{\lbeta, K}, \dots, S_{f, p}^{\lbeta, K})$ by a neural network with ReQU activations, we first perform this task for the products of basis functions $\prod_{\ell = 1}^d B_{j_\ell}^{\lbeta,K}(x_\ell)$, $1 \leq j_1, \dots, j_d \leq \lbeta + K$.

\medskip

\noindent{\bf Step 2.}\quad
Applying \Cref{lem:b-spline_representation} with $q = \lbeta$ component-wise for each $x_i, i \in \{1,\ldots,d\}$, we obtain that the mapping 
\begin{align}
    \label{eq:splines_costruct}
    x = (x_1,\dots,x_d) \mapsto \Big(
    &\notag
    x_1, K, B_{1}^{\lbeta, K}(x_1), \dots, B_{\lbeta + K}^{\lbeta, K}(x_1),
    \\&
    x_2, K, B_{1}^{\lbeta, K}(x_2), \ldots, B_{\lbeta + K}^{\lbeta, K}(x_2), \dots,
    \\&\notag
    x_d, K, B_{1}^{\lbeta, K}(x_d), \ldots, B_{\lbeta + K}^{\lbeta, K}(x_d) \Big)
\end{align}
can be represented by a neural network from the class $\NN\big(4 + 2(\lbeta-2), (d, d\A_1) \big)$, where
\begin{equation}
\label{eq:A_1_def}
    \A_1
    =
    \big(\mathcal{B}_{2}, \mathcal{B}_{3}, \ldots, \mathcal{B}_{\lbeta}, K + \lbeta + 2\big),
\end{equation}
and $\mathcal{B}_{2}, \dots, \mathcal{B}_{\lbeta}$ are defined in \eqref{eq:architecture_spline_third_degree} and \eqref{eq:B_spline_m_order}.
Here $d\A_1$ should be understood as element-wise multiplication of $\A_1$ entries by $d$. Since each coordinate transformation in \eqref{eq:splines_costruct} can be computed independently, the number of parameters in this network does not exceed
\[
    72 d \lbeta (K+2\lbeta ).
\] 
Using \Cref{lem:requ_prod} with $k = d$, for any fixed $j_1,\dots,j_d \in \{1,\dots,\lbeta + K\}$, we calculate all products $\prod_{\ell=1}^{d}B_{j_\ell}^{\lbeta,K}(x_\ell)$ by a neural network from
\[
    \NN\left(\lceil \log_{2}{d} \rceil,\big( (K + \lbeta + 2) d, (K + \lbeta)^{d} \A_2, (K + \lbeta)^{d})\right)
\]
with
\begin{equation}
    \label{eq:A_2_def}
    \A_2 = \bigl(\underbrace{2^{\lceil \log_{2}{d} \rceil + 1}, 2^{\lceil \log_{2}{d} \rceil }, \ldots, 4}_{\lceil \log_{2}{d} \rceil}\bigr)\,.
\end{equation}
Note that this network contains at most 
\[
    5(K+\lbeta)^{d}2^{2\lceil \log_{2}{d} \rceil} \leq 20 d^2(K+\lbeta)^{d}
\]
non-zero parameters. 

\medskip

\noindent{\bf Step 3.}\quad
Let us recall that, due to \eqref{eq:spline_expansion}, we have
\begin{equation}
    \label{eq:spline_exp_fin}
    S_{f, m}^{\lbeta, K}
    = \sum\limits_{j_1, \dots, j_d = 1}^{\lbeta + K} w^{(f)}_{m, j_1, \dots, j_d} \prod\limits_{\ell = 1}^{d} (a_{j_\ell + \lbeta + 1} - a_{j_\ell}) \prod\limits_{\ell = 1}^{d} B_{j_\ell}^{\lbeta, K}(x_\ell)
\end{equation}
for all $m \in \{1, \dots, p\}$.
On Step $2$, we constructed a network which calculates the products $\prod\limits_{\ell = 1}^{d} B_{j_\ell}^{\lbeta,K}(x_\ell)$ for all $(j_1, \dots, j_d) \in \{1, \dots,\lbeta+K\}^d$. Now we implement multiplications by 
\[
    \tilde{w}^{(f)}_{m,j_1, \dots, j_d} := w^{(f)}_{m,j_1, \dots, j_d} \prod\limits_{\ell = 1}^{d} (a_{j_\ell+\lbeta+1} - a_{j_\ell}).
\]
\Cref{thm:coef_bound} implies that
\[
    \left|w^{(f)}_{m,j_1, \dots, j_d} \right|
    \leq	(2\lbeta + 1)^d 9^{d(\lbeta - 1)} \|f_m\|_{L_\infty([0,1]^d)} \leq (2\lbeta + 1)^d 9^{d(\lbeta - 1)} H
\]
Since $\beta > 2$ by the conditions of the theorem, we can write
\[
    \left|w^{(f)}_{m,j_1, \dots, j_d} \right|
    \leq (2\lbeta + 1)^{d \lbeta} H.
\]
Equation \eqref{eq:knots} yields that $0 \leq a_{j+\lbeta+1} - a_j \leq 1$ for all $j \in \{1,\ldots,\lbeta + K\}$. In view of \Cref{lem:requ_mult}, we can implement the multiplication by $\tilde{w}^{(f)}_{m,j_1, \dots, j_d}$ by a network from the class
\[
    \NN\left( 2\left( \left\lceil \log_2 (2d\lbeta + d) \vee \log_2 \log_2 H \right\rceil \vee 1 \right) + 2, \big(1, 5, 5, \dots, 5, 4, 1 \big) \right).
\]
Here we used the fact that, since $\lbeta \geq 2$,
\begin{align*}
    \log_4 \left( \log_4 \big( (2\lbeta + 1)^{d \lbeta} H \big) \right)
    &
    \leq \log_4 \left( 2\log_4 \big( (2\lbeta + 1)^{d \lbeta} \big) \vee 2\log_4 H \right)
    \\&
    = \log_4 \left( \log_2 \big( (2\lbeta + 1)^{d \lbeta} \big) \vee \log_2 H \right)
    \\&
    =  \log_4 \left( d \lbeta \log_2 \big( (2\lbeta + 1) \big) \right) \vee \log_4 \left( \log_2 H \right)
    \\&
    \leq \log_2 (2d\lbeta + d) \vee \log_2 \log_2 H.
\end{align*}
Hence, the representation \eqref{eq:spline_exp_fin} implies that the function $S_f^{\lbeta, K}$ can be represented by a neural network with
\[
    4 + 2(\lbeta-2) + \lceil \log_{2}{d} \rceil + 2\left( \left\lceil \log_2 (2d\lbeta + d) \vee \log_2 \log_2 H \right\rceil \vee 1 \right) + 2
\]
hidden layers and the architecture
\[
    \left( d, d\A_1, (K + \lbeta)^{d}\A_2, \underbrace{5(K + \lbeta)^{d}, 5(K + \lbeta)^{d}, \dots, 5(K + \lbeta)^{d}}_{\text{$2\left( \left\lceil \log_2 (2d\lbeta + d) \vee \log_2 \log_2 H \right\rceil \vee 1 \right) + 1$ times}}, 4(K + \lbeta)^{d}, p \right),
\]
where $\A_1$ and $\A_2$ are given by \eqref{eq:A_1_def} and \eqref{eq:A_2_def}, respectively. As before, weights of the constructed neural network are bounded by $1$.
Note that this network contains at most
\begin{align*}
    &
    p (K + \lbeta)^{d} \left( 60 \left( \left\lceil \log_2 (2d\lbeta + d) \vee \log_2 \log_2 H \right\rceil \vee 1 \right) + 38 \right)
    \\&\quad
    + 20 d^2(K+\lbeta)^{d}
    + 72 d \lbeta (K + 2\lbeta )
    + 8 d (K + \lbeta)^{d}
\end{align*}
non-zero weights. The last summand appears, because each of 
\[
    (K + \lbeta)^d 2^{\lceil \log_{2}{d} \rceil + 1} \leq 4d (K + \lbeta)^d
\]
neurons in the first layer of $(K + \lbeta)^d \A_2$ receives information from two neurons from the previous layer.
This completes the proof.
	
\subsection{Auxiliary lemmas for \Cref{prop:requ_approx}}
\begin{Lem}
\label{lem:requ_prod}
Let $k \in \nset, k \geq 2$. Then, for any $x = (x_1,\dots,x_k) \in \R^k$, there exists a neural network from the class
\begin{equation}
\label{eq:NN_class_product}
\NN\left(\lceil \log_{2}{k} \rceil , \left(k, 2^{\lceil \log_{2}{k} \rceil + 1}, 2^{\lceil \log_{2}{k} \rceil}, \dots, 4, 1\right)\right)\,,
\end{equation}
which implements the map $x \mapsto x_1x_2,\dots x_k$. Moreover, this network contains at most $5 \cdot 2^{2\lceil \log_{2}{k} \rceil}$ non-zero weights.
\end{Lem}
\begin{proof}[Proof of \Cref{lem:requ_prod}] Let $x = (x_1,x_2) \in \R^2$. We first prove that the map $(x_1, x_2) \mapsto x_1 x_2$ belongs to the class $\NN(1, (2, 4, 1))$. Let us define $W_1 = \frac{1}{4}(1,-1,-1,1)^{\top}$ and
\begin{equation*}
W_0 = 
\begin{pmatrix}
    1 & 1 \\
    1 & -1 \\
    -1 & 1 \\
    -1 & -1
\end{pmatrix}\,.
\end{equation*}
Then it is easy to see that
\begin{equation}
\label{eq:product_of_two_vars}
\begin{split}
    W_1 \circ \requ \circ W_0 x
    &
    \equiv \frac{1}{4} \left(\requ(x_1 + x_2) + \requ(-x_1 - x_2) - \requ(x_1 - x_2) - \requ(x_2-x_1) \right)
    \\&
    = \frac{1}{4} \left((x_1 + x_2)^2 - (x_1 - x_2)^2\right)
    = x_1 x_2.
\end{split}
\end{equation}
Now, for any $k \geq 2$, we use the following representation
\begin{equation}
\label{eq:product_representation}
x_1 x_2 \dots x_{k} = x_1 x_2 \dots x_{k} \cdot \underbrace{1 \cdot \ldots \cdot 1}_{2^{\lceil \log_{2}{k} \rceil}-k}\,.
\end{equation}
Put $v = \lceil \log_{2}{k} \rceil$. Then it remains to note that, based on the equality \eqref{eq:product_of_two_vars}, for any $v \in \nset$, we can implement a sequence of mappings
\[
    \left(a_1,a_2,\ldots,a_{2^v}\right) \mapsto \left(a_1a_2,a_3a_4,\ldots,a_{2^v-1}a_{2^v}\right) \mapsto \ldots \mapsto \prod_{i=1}^{2^v}a_i
\]
by a neural network from the class
\[
    \NN\left(v, \left(2^{v}, 2^{v + 1}, 2^v, 2^{v-1}, 2^{v-2}, \dots, 4, 1\right)\right)\,.
\]
The number of parameters in such network does not exceed $2^{2v + 1} + 2^{2v + 1} + 2^{2v} + \ldots + 4 \leq 5 \cdot 2^{2v}$. We complete the proof, combining this bound with \eqref{eq:product_representation}.

\end{proof}

\begin{Lem}
\label{lem:b-spline_representation_quadratic}
Let $q, K $ be integers not smaller than $2$, and $x \in [0, 1]$. Then the mapping 
\[
    x \mapsto \left(x,K,B_1^{2, K}(x),B_2^{2, K}(x),\ldots,B_{K+2q-2}^{2, K}(x)\right),
\]
can be represented by a network from the class $\NN \Big(4,\big(1, \mathcal{B}_{2}, K + 2q \big) \Big)$, where
\begin{equation}
\label{eq:architecture_spline_third_degree}
\mathcal{B}_{2} = \left(4(K+2q-1) + K, 4K + 8q, 4K + 8q, 4K + 8q\right),
\end{equation}
containing at most $72(K + 2q)$ non-zero weights.
\end{Lem}

\begin{proof}
Note first that, due to \citep[Theorem 4.32]{schumaker07},
\begin{align*}
    B_j^{2, K}(x)
    &
    = \frac{K^3}6 
    \left( \left(x - \frac{j - q - 1}K \right)_+^2 - 3\left(x - \frac{j - q}K \right)_+^2 \right.
    \\&\quad
    \left. + 3\left(x - \frac{j - q + 1}K \right)_+^2 - \left(x - \frac{j - 
    q + 2}K \right)_+^2  \right), \quad q + 1 \leq j \leq q + K - 2.
\end{align*}
The functions $B_{q-1}^{2, K}$, $B_q^{2, K}$, $B_{q + K - 1}^{2, K}$, and $B_{q + K}^{2, K}$ can be computed directly using \eqref{eq:b-spline_recursion}. Indeed, for any $x \in [0, 1]$, 
\begin{align*}
B_{q-1}^{2, K}(x) &= K^3 \left( \frac1K - x \right)_+^2,\\
B_q^{2, K}(x) &= \frac{K^3}4 \left( \left( \frac2K - x \right)_+^2 - 4 \left( \frac1K - x \right)_+^2 + 
	3(-x)_+^2 \right),
	\\
	B_{q+K-1}^{2, K}(x)
	&
	= \frac{K^3}4 \left( \left( x  - \frac{K-2}K \right)_+^2 - 2 \left( x - \frac{K-1}K 
	\right)_+^2 - 3(x - 1)_+^2 \right),
	\\
	B_{q+K}^{2, K}(x)
	&
	= K^3 \left( x - \frac{K-1}K \right)_+^2\,.
\end{align*}
Moreover, \eqref{eq:b-spline_recursion} implies that $B_{1}^{2, K}(x) = \dots = B_{q-2}^{2, K}(x) = 0$ and  $B_{q+K+1}^{2, K}(x) = \ldots = B_{2q+K-2}^{2, K}(x) = 0$. Hence, each of the functions $B_j^{2, K}/K^3, j \in \{1,\dots,2q+K-2\}$ can be exactly represented by a neural network from the class $\NN(1, (1, 4, 1))$.
The final mapping (b) is implemented as $ K^3 = \frac14 (K^2 + K)^2 - \frac14 (K^2 - K)^2$. Note that the identity map $x \mapsto x$ belongs to $\NN(1, (1, 4, 1))$ too:
\begin{equation}
\label{eq:ident_map_repr}
x = \left( (x+1)^2 - (x-1)^2 \right)/4 = \left( \requ_{-1}(x) + \requ_{-1}(-x) - \requ_1(x) - \requ_1(-x)\right)/4.
\end{equation}
Combining the arguments above, we conclude that the mapping 
\begin{equation}
\label{eq:map_norm_splines}
x \mapsto \left(x,B_1^{2, K}(x)/K^3,B_2^{2, K}(x)/K^3,\ldots,B_{K+2q-2}^{2, K}(x)/K^3\right)
\end{equation}
can be exactly represented by a network from the class $\NN(1, (1, 4(K+2q-1), K+2q-1))$.
Besides, one can implement the sequence of transformations
\begin{equation}
    \label{eq:map_k}
    x
    \mapsto \underbrace{(1,\dots,1)}_{K \text{times}} \mapsto K,
\end{equation}
using a neural network from $\NN\big(1, (1, K, 1 \big)$.
Connecting in parallel the neural networks realizing the maps \eqref{eq:map_norm_splines} and \eqref{eq:map_k}, we obtain that one can implement the transformation
\begin{equation}
    \label{eq:map_k_norm_splines}
    x \mapsto \left(x, K, B_1^{2, K}(x)/K^3,B_2^{2, K}(x)/K^3,\ldots,B_{K+2q-2}^{2, K}(x)/K^3\right)
\end{equation}
via a neural network from the class $\NN\big(1, (1, 4(K+2q-1) + K, K+2q), 13(K+2q-1) + 3K + 1\big)$.

It remains to implement multiplication of the splines by $K^{3}$ to complete the proof. Recall that, according to Lemma \ref{lem:requ_prod}, the function, mapping a pair $(x, y)$ to their product $xy$, belongs to the class $\NN(1, (2, 4, 1))$. 
For any $x \in \R$, consider a vector
\[
    \left(x, K, B_1^{2, K}(x)/K^3,B_2^{2, K}(x)/K^3,\ldots,B_{K+2q-2}^{2, K}(x)/K^3\right).
\]
Let us implement the multiplications of the splines $B_1^{2, K}(x)/K^3,B_2^{2, K}(x)/K^3, \ldots, B_{K+2q-2}^{2, K}(x)/K^3$ by $K$ in parallel. Then we obtain that the mapping
\begin{align*}
    &
    \left(x, K, B_1^{2, K}(x)/K^3,B_2^{2, K}(x)/K^3,\ldots,B_{K+2q-2}^{2, K}(x)/K^3\right)
    \\&
    \mapsto \left(x, K, B_1^{2, K}(x)/K^2,B_2^{2, K}(x)/K^2,\ldots,B_{K+2q-2}^{2, K}(x)/K^2\right)
\end{align*}
is in the class $\NN\big(1, (K + 2q, 4K + 8q, K + 2q), 17(K + 2q) - 8 \big)$. Here we took into account that we need $17$ non-zero weights to implement each of $(K + 2q - 2)$ multiplications and $13$ non-zero weights to implement the identity maps $x \mapsto x$ and $K \mapsto K$. Repeating the same trick two more times, we get that there is a neural network in $\NN\big(3, (K + 2q, 4K + 8q, 4K + 8q, 4K + 8q, K + 2q), 57(K + 2q) - 8 \big)$,
implementing 
\begin{equation}
    \label{eq:map_k_cube_mult}
    \begin{split}
    &
    \left(x, K, B_1^{2, K}(x)/K^3,B_2^{2, K}(x)/K^3,\ldots,B_{K+2q-2}^{2, K}(x)/K^3\right)
    \\&
    \mapsto \left(x, K, B_1^{2, K}(x),B_2^{2, K}(x),\ldots,B_{K+2q-2}^{2, K}(x)\right).
    \end{split}
\end{equation}
This follows from the fact that each of $2K + q - 2$ multiplications by $K^3$, running in parallel, can be implemented by a neural network from $\NN(3, (2, 4, 4, 4, 1))$, and two identity maps $x \mapsto x$ and $K \mapsto K$ can be represented by a network from $\NN(3, (1, 4, 4, 4, 1))$. 
Concatenating the neural networks, performing \eqref{eq:map_k_norm_splines} and \eqref{eq:map_k_cube_mult}, we finally get that
the mapping of interest,
\[
    x \mapsto \left(x,K,B_1^{2, K}(x),B_2^{2, K} (x), \ldots, B_{K+2q-2}^{2, K}(x)\right),
\]
is in the class of neural networks from $\NN\big(3, (1, 4(K+2q-1) + K, 4K + 8q, 4K + 8q, 4K + 8q, K + 2q)\big)$, containing at most
\[
    69(K + 2q) - 20 + 3K 
    < 72(K + 2q)
\]
non-zero weights.
\end{proof}

\begin{Lem}
\label{lem:b-spline_representation} Let $x \in \R$ and $K,q \in \nset, K,q \geq 2$. Then for any $j \in \{1,\dots,q+K\}$, the mapping $x \mapsto \left(x,K,B_1^{q, K}(x),B_2^{q, K}(x),\ldots,B_{q+K}^{q, K}(x)\right)$ belongs to the class
\begin{equation*}
    \NN\Big(4 + 2(q-2),
    \big(1, \mathcal{B}_{2}, \mathcal{B}_{3}, \ldots, \mathcal{B}_{q}, K + q + 2\big)\Big)\,,
\end{equation*}
with at most $72q(K+2q)$ non-zero weights. The vector $\mathcal{B}_{2}$ is defined in \eqref{eq:architecture_spline_third_degree}, and for $m \in \{3,\ldots,q\}$, $\mathcal{B}_{m}$ is equal to 
\begin{equation}
\label{eq:B_spline_m_order}
\mathcal{B}_{m} = \bigl(12(K + 2q - m) + 12, 8(K + 2q - m) + 8\bigr)\,.
\end{equation}
\end{Lem}

\begin{proof}[Proof of \Cref{lem:b-spline_representation}] Fix an integer $q \geq 2$. We consider the family of $B$-splines $\bigl\{B_j^{m, K}(x), m \in \{2,\dots,q\}, j \in \{1,\dots,K+2q-m\}\bigr\}$ and construct all elements of this family sequentially starting from $m = 2$. By \Cref{lem:b-spline_representation_quadratic}, the map 
\[
x \mapsto \left(x,K,B_1^{2, K}(x),B_2^{2, K}(x),\ldots,B_{K+2q-2}^{2, K}(x)\right)
\]
can be represented by a network in the class \eqref{eq:architecture_spline_third_degree}, containing at most $72(K+2q)$ non-zero weights. Assume that for some $m \geq 3$ the mapping
\[
    x \mapsto \left(x,K,B_1^{m-1, K}(x),\ldots,B_{K+2q-(m-1)}^{m-1, K}(x)\right)
\]
belongs to the class 
\begin{align}
\label{eq:nn_m_th_order_splines}
    \NN\Big(4 + 2(m-3),
    \big(1, \mathcal{B}_{2}, \mathcal{B}_{3}, \ldots, \mathcal{B}_{m-1}, K+2q-m+3 \big)\Big).
\end{align}
We use the recursion formula \eqref{eq:b-spline_recursion} to perform the induction step.
Note that we can implement each of the mappings
\begin{equation}
    \label{eq:aux_products}
    (x,K) \mapsto \frac{x-a_j}{a_{j+m+1}-a_j}\,,
    \quad
    (x,K) \mapsto \frac{a_{j+m+1} - x}{a_{j+m+1}-a_j}\,,
    \quad
    j \in \{1, \dots, 2q + K - m\}\,,
\end{equation}
by networks from $\NN(1,(2,4,1))$. It is possible, since for $j$ and $m$ satisfying $a_{j+m+1} - a_{j} > 0$, it holds that $a_{j+m+1} - a_{j} \geq 1/K$. We remove the last linear layer of \eqref{eq:nn_m_th_order_splines}, and concatenate the remaining part with the first layer of networks, implementing \eqref{eq:aux_products}. We obtain a network with $5+2(m-3)$ hidden layers and the architecture 
\begin{equation}
\label{eq:product_spline_net}
\NN\Big(5 + 2(m-3),
    \big(1, \mathcal{B}_{2}, \mathcal{B}_{3}, \ldots, \mathcal{B}_{m-3},
    12(K+2q-m) + 12, 3(K+2q-m)+3 \big)\Big)\,,
\end{equation}
which implements the mapping
\begin{equation}
\label{eq:intermediate_nn_lem_15}
x \mapsto \left(x,K,B_1^{m-1, K}(x),\ldots,B_{K+2q-(m-1)}^{m-1, K}(x), \underbrace{\frac{x-a_j}{a_{j+m+1}-a_j}, \frac{a_{j+m+1} - x}{a_{j+m+1}-a_j}}_{j \in \{1, \dots, 2q + K - m\}} \right)\,.
\end{equation}
Note that we have added at most 
\[
\underbrace{4(K+2q-m+3)}_{\text{to implement identity maps}} + \underbrace{16 \cdot 2(K+2q-m)}_{\text{to implement \eqref{eq:aux_products}}} = 36(K+2q-m) + 12
\] 
non-zero weights, since each linear function of the form \eqref{eq:aux_products} can be implemented via a network from $\NN\big(1,(2,4,1)\big)$ with at most $16$ non-zero weights. According to \eqref{eq:b-spline_recursion}, for any $j \in \{1, \dots, 2q + K - m\}$, we construct
\begin{equation}
\label{eq:b-spline_recursion-main}
B_j^{m, K}(x) =
\begin{cases}
    \frac{(x - a_j)B_j^{m-1, K}(x) + (a_{j+m+1} - x)B_{j+1}^{m-1, K}(x)}{a_{j+m+1} - 
	a_j},\\
	\qquad \qquad \qquad \qquad \qquad \text{if $a_j < a_{j+m+1}, a_j \leq x < a_{j+m+1}$},\\
    0, \quad \text{otherwise}.
\end{cases}
\end{equation}
Now we add one more hidden layer with $8(K + 2q - m) + 8$ parameters to the network \eqref{eq:intermediate_nn_lem_15}, yielding a network with an architecture
\begin{equation}
\label{eq:b_spline_intermediate}
\begin{split}
& \NN\Big(6 + 2(m-3),
    \big(1, \mathcal{B}_{2}, \mathcal{B}_{3}, \ldots, \mathcal{B}_{m-1},
    12(K+2q-m) + 12, \\
&\qquad \qquad \qquad \qquad \qquad 8(K + 2q - m) + 8, 2(K+2q-m)+2 \big)\Big)\,,
\end{split}
\end{equation}
which implements
\[
x \mapsto \left(x,K, \underbrace{\frac{x-a_j}{a_{j+m+1}-a_j}B_j^{m-1, K}(x), \frac{a_{j+m+1} - x}{a_{j+m+1}-a_j}B_{j+1}^{m-1, K}(x)}_{j \in \{1, \dots, 2q + K - m\}} \right).
\]
Note that adding the new hidden layer adds at most 
\[
    \underbrace{8}_{\text{to implement identity maps for $x$ and $K$}} + \underbrace{16 \cdot 2(K+2q-m)}_{\text{to implement \eqref{eq:b-spline_recursion}}}
\]
additional non-zero-weights. Combining the above representations and \eqref{eq:b-spline_recursion-main}, we obtain a network from the class
\begin{equation*}
\begin{split}
    \NN\Big(4 + 2(m-2),
    &
    \big(1, \mathcal B_2, \mathcal B_3, \dots, \mathcal B_m, K+2q-m+2 \big)\Big)
\end{split}
\end{equation*}
with at most
\begin{align*}
    &
    72(K + 2q) + \sum\limits_{s=3}^{m} \Big[68(K + 2q - s) + 20\Big]
    \\&
    \leq 72(K + 2q) + (m - 2) \Big[68(K + 2q) + 20\Big]
    \\&
    \leq 72(K + 2q) + (m - 2) \cdot 72 (K + 2q)
    \\&
    < 72 m (K + 2q)
\end{align*}
non-zero weights.
\end{proof}

\begin{Lem}
    \label{lem:requ_mult}
    Let $x \in \R$ and let $L \in \mathbb N$. Then, for any $M$, such that $|M| \leq 4^{4^L}$, the mapping $x \mapsto Mx$ can be represented by a neural network, belonging to the class
    \[
        \NN\Big( 2L + 2, (1, 5, 5, \dots, 5, 4, 1) \Big)
    \]
    and containing at most $60 L + 38$ non-zero weights.
\end{Lem}

\begin{proof}
    Using the representation $4 = (1 + 1)^2$,
    we obtain that there is a neural network from $\NN((2L + 1, (1, 1, \dots, 1) \big)$, implementing the sequence of maps
    \[
        x \mapsto 1 \mapsto 4 \mapsto 4^2 \mapsto 4^4 \mapsto \dots \mapsto 4^{4^L}.
    \]
    At the same time, due to
    \[
        x
        = \frac14 \left( (x - 1)^2 - (x + 1)^2 \right)
        = \frac14 \left( (x - 1)_+^2 - (1 - x)_+^2 - (x + 1)_+^2 + (-x - 1)_+^2 \right),
    \]
    the identity map $x \mapsto x$ belongs to
    \[
        \NN((2L + 1, (1, \underbrace{4, 4, \dots, 4}_{\text{$2L + 1$ times}}, 1) \big).
    \]
    Running the two neural networks in parallel, we can implement the mapping
    \[
        x \mapsto \left(4^{4^L}, x \right)
    \]
    via a neural network from $\NN\big( 2L + 1, (1, 5, 5, \dots, 5, 2)\big)$. Finally, applying \Cref{lem:requ_prod}, we conclude that there is a neural network from the class
    \[
        \NN\big( 2L+2, (1, 5, 5, \dots, 5, 4, 1) \big),
    \]
    performing the map $x \mapsto 4^{4^L} x$. The number of non-zero weights in this network does not exceed
    \[
        \left( 1 \cdot 5 + 5 \cdot 5 \cdot 2L + 5 \cdot 4 + 4 \cdot 1 \right) + 5 \cdot (2L + 1) + 4 = 60 L + 38.
    \]
    It remains to note that, by the definition of $L$, we have $|M| / 4^{4^L} \leq 1$.
    
\end{proof}

\subsection{Proof of \Cref{co:analyt_exp_bound}}

Let $s > \ell$ be an integer to be specified later. Plugging the bound $\left\|f\right\|_{\H^s([0, 1]^d)} \leq Q R^{-s} s!$, $s \in \mathbb N_0$, into \eqref{eq:approx_H_l_norm}, we get that there is a neural network $h_f$ of width
\[
    \left(4d (K + s - 1)^d \right) \vee \left( 12(K + 2s - 2) + 12\right) \vee p
\]
with
\[
    6 + 2(s - 3) + \left\lceil \log_{2}{d} \right\rceil + 2\left( \left\lceil \log_2 (2d(s - 1) + d) \vee \log_2 \log_2 (Q R^{-s} s!) \right\rceil \vee 1 \right)
\]
hidden layers and at most 
\begin{align*}
    &
    p (K + s - 1)^{d} \left( 60 \left[ \left\lceil \log_2 (2d(s - 1) + d) \vee \log_2 \log_2 (Q R^{-s} s!) \right\rceil \vee 1 \right] + 38 \right)
	\\&
	+ p (K + s - 1)^{d} \left( 20 d^2 + 144 d(s-1) + 8 d \right)
\end{align*}
non-zero weights, taking their values in $[-1, 1]$, such that, for any $s > \ell$,
\begin{align*}
    \left\|f - h_f\right\|_{\H^\ell([0, 1]^d)}
	&
	\leq  \frac{ Q (\sqrt{2}e d)^s s!}{R^s K^{s-\ell}} + \frac{Q 9^{d(s - 2)} (2s - 1)^{2d + \ell} (\sqrt{2}ed)^s s!}{R^s K^{s - \ell}}
	\\&
	\leq \frac{Q 9^{ds} (2s - 1)^{2d + \ell} (\sqrt{2}ed)^s s!}{R^s K^{s - \ell}}.
\end{align*}
Recall that, by the definition, $\lfloor s\rfloor = s - 1$ for any integer $s$. Using the inequalities
\[
    s! \leq (s - \ell)! \, s^\ell,
    \quad
    (s - \ell)! \leq e \sqrt{s - \ell} \left( \frac{s - \ell}e \right)^{s - \ell},
    \quad \text{and} \quad
    \sqrt{s - \ell} \leq 2^{(s - \ell) / 2},
\]
which are valid for all $s > \ell$, we obtain that
\begin{equation}
\label{eq:analyt_bound_pre-final}
\begin{split}    
    \left\|f - h_f\right\|_{\H^\ell([0, 1]^d)}
	&
	\leq \frac{Q e (\sqrt{2} \, e d \, 9^d)^\ell (2s - 1)^{2d + \ell} s^\ell}{R^\ell} \cdot \left( \frac{2d \, 9^d (s - \ell)}{K R} \right)^{s - \ell}
	\\&
	\leq \frac{Q e (\sqrt{2} \, e d \, 9^d)^\ell (2s - 1)^{2d + 2\ell}}{R^\ell} \cdot \left( \frac{2d \, 9^d (s - \ell)}{K R} \right)^{s - \ell}.
\end{split}
\end{equation}
Set 
\begin{equation}
\label{eq:s_choice}
    s = \ell + \left\lfloor \frac{KR}{2ed \, 9^{d}} \right\rfloor.
\end{equation}
It is easy to observe that, with such a choice of $s$, the neural network $h_f$ is of the width $O((K + \ell)^d)$ and has $O((K + \ell))$ hidden layers and at most $O( p (K + \ell)^{d} \log K)$ non-zero weights. The bound \eqref{eq:analyt_bound} on the approximation error of $h_f$ follows directly from \eqref{eq:analyt_bound_pre-final} and \eqref{eq:s_choice}.

\bibliography{references}

\appendix

\section{Some properties of multivariate splines}
\label{app:splines}
In this section we provide some properties of multivariate splines. For more details we refer reader to the book \cite{schumaker07}.

\subsection{Univariate splines}
\label{sec:univariate_splines}
We start with one-dimensional case. Fix $K,q \in \nset$. We call a function $S : [0, 1] \mapsto \R$ a univariate spline of degree $q$ with knots at $\{1/K, 2/K, \dots, 
(K-1)/K\}$, if
\begin{itemize}
\item for any $i \in \{0, 1, \dots, K-1\}$, the restriction of $S$ on $[i/K, (i+1)/K]$ is a polynomial of degree $q$;
\item $S$ has continuous derivatives up to order $q-1$, that is, for any $m \in \{0,\dots, q-1\}$ and any $i \in \{1, 2, \dots, K-1\}$, $D^m S(i/K + 0) = D^m S(i/K - 0)$.
\end{itemize}
It is easy to observe that univariate splines of degree $q$ with knots at $\{1/K, 2/K, \dots, 
(K-1)/K\}$ form a linear space.
We denote this space by $\splinespace_{q, K}$. Note that $\splinespace_{q, K}$ has a finite dimension $q + K$. There are several ways to choose a basis in $\splinespace_{q,K}$, below we construct the basis of normalized $B$-splines $\{N_j^{q, K}(x) : 1 \leq j \leq K + q\}$. Let $a = (a_1,\dots,a_{2q + K + 1})$ be a vector of knots defined in \eqref{eq:knots}. First, we successively construct (unnormalized) B-splines $\{B_j^{q, K}(x) : 1 \leq j \leq K + q\}$ associated with knots $a_1, \dots, a_{2q+K+1}$. For any $j \in \{1, \dots, 2q+K\}$, let
\[
B_j^{0, K}(x) =
\begin{cases}
1/(a_{j+1} - a_j), \quad \text{if $a_j < a_{j+1}$ and $a_j \leq x < a_{j+1}$},\\
0, \quad \text{otherwise}.
\end{cases}
\]
Then for any $m \in \{1,\dots,q\}$, $j \in \{1,2q+K-m\}$, we put
\begin{equation}
\label{eq:b-spline_recursion}
B_j^{m, K}(x) =
\begin{cases}
    \frac{(x - a_j)B_j^{m-1, K}(x) + (a_{j+m+1} - x)B_{j+1}^{m-1, K}(x)}{a_{j+m+1} - 
	a_j},\\
	\qquad \qquad \qquad \qquad \qquad \text{if $a_j < a_{j+m+1}, a_j \leq x < a_{j+m+1}$},\\
    0, \quad \text{otherwise}.
\end{cases}
\end{equation}
Now, for $m \in \{1,\dots,q\}$, we define the normalized $B$-spline $N_j^{m, K}$ as
\begin{equation}
\label{eq:norm_b-spline}
N_j^{m, K}(x) = (a_{j+m+1} - a_j) B_j^{m, K}(x), \quad 1 \leq j \leq 2q + K - m.
\end{equation}
We use the following properties of normalized $B$-splines (see \cite[Section 4.3]{schumaker07} for the details).

\begin{Prop}[\cite{schumaker07}, Section 4.3]
	\label{prop:b-spline}
	Let $N_j^{m, K}$ be a normalized B-spline defined in \eqref{eq:norm_b-spline}. Then the following holds.
	\begin{itemize}
		\item For any $q \in \nset_{0}, j \in \{1, \dots, q+K\}$, 
		\[
		    \text{\rm supp}(N_j^{q, K}) = [a_j, a_{j+q+1}]
		    \quad \text{and} \quad
		    \|N_j^{q, K}\|_{L_\infty([0,1])} \leq 1.
		\]
		\item For any $q \in \mathbb N$ and any $j \in \{1, \dots, q+K\}$, $N_j^{q, K}$ is a 
		spline of degree $q$, it has continuous derivatives up to order $q-1$ and 
		the $(q-1)$-th derivative is Lipschitz.
		Moreover, for any $m \in \{1, \dots, q\}$,
		\[
		\frac{\rmd N^{m, K}_j(x + 0)}{\rmd x} = \frac{m}{a_{j+m} - a_j} \left( 
		N_j^{m-1, K}(x) - N_{j+1}^{m-1, K}(x) \right),
		\quad 1 \leq j \leq 2q+K-m.
		\]
	\end{itemize}
\end{Prop}
\begin{Co}
\label{co:b-spline_deriv}
For any $m \in \mathbb N$ and any $\ell \in \{1, \dots, m\}$, 
\[
\max\limits_{1 \leq j \leq 2q + K - m} \left\| D^\ell N^{m, K}_j \right\|_{L_\infty([0,1])} \leq (2K)^\ell \frac{m!}{(m-\ell)!}.
\]
\end{Co}
\begin{proof}
\Cref{prop:b-spline} implies that, for any $\ell \in \{1, \dots, m\}$,
\begin{align*}
&
\max\limits_{1 \leq j \leq 2q + K - m} \left\| D^\ell 
N^{m, K}_j \right\|_{L_\infty([0,1])}
\\&
= \max\limits_{1 \leq j \leq 2q + K - m} \frac{m}{a_{j+m} - a_j} 
\left\|D^{\ell-1} N_j^{m-1, K} - D^{\ell-1} N_{j+1}^{m-1, K} \right\|_{L_\infty([0,1])}
\\&
\leq 2mK \max\limits_{1 \leq j \leq 2q + K - m + 1} \left\| D^{\ell-1} N_j^{m-1, K} 
\right\|_{L_\infty([0,1])}.
\end{align*}
Now the statement follows from the induction in $\ell$ together with $\|N_j^{q, K}\|_{L_\infty([0,1])} \leq 1$. 
\end{proof}

To study approximation properties of $\splinespace_{q, K}$ with respect to the H\"older 
norm $\|\cdot\|_{\H^k}, k \in \nset_{0}$, we follow the 
technique from \cite[Sections 4.6 and 12.3]{schumaker07} and introduce a dual basis for normalized B-splines. A set of linear functionals $\{\lambda^{q, K}_1, \dots, \lambda^{q, K}_{q+K}\}$ on $L_\infty([0,1])$ is called a dual basis if, for all $i, j \in \{1, \dots, q+K\}$,
\[
\lambda_i^{q, K} N_j^{q, K} = 
\begin{cases}
1, \quad \text{if $i = j$,}\\
0, \quad \text{otherwise}.
\end{cases}
\]
An explicit expression for $\lambda_i^{q, K}$ can be found in \cite[Eq.(4.89)]{schumaker07} but it is not necessary for our purposes. Define a linear operator $\mathcal Q^{q, K} : L_\infty([0, 1]) \rightarrow \mathcal 
S_{q, K}$ by 
\[
\mathcal Q^{q, K} f = \sum\limits_{j = 1}^{q + K} (\lambda_j^{q, K} f) N_j^{q, K}.
\]
The function $\mathcal Q^{q, K} f$ is called quasi-interpolant and it nicely approximates 
$f$ provided that $f$ is smooth enough. The approximation properties of $\mathcal Q^{q, K} f$ are studied below for the case of multivariate tensor-product splines.

\subsection{Tensor-product splines} The main result of this section is \Cref{thm:spline_approx} concerning 
approximation properties of multivariate splines.
We start with auxiliary definitions.
The tensor product
\[
\splinespace^{q, K}_d = \left(\splinespace^{q, K}\right)^{\otimes d}
\equiv \text{span}\left\{ N_{j_1}^{q, K}(x_1)\cdot \dots \cdot N_{j_d}^{q, K}(x_d) : j_1, \dots, 
j_d \in \{1, \dots, q + K\} \right\}
\]
is called a space of tensor-product splines of degree $q$.
Since $q$ and $K$ are fixed, we omit the upper indices in the notations $B_j^{q, K}, N_j^{q, K}$, $1 \leq j \leq q + K$, if there is no ambiguity. For any $i \in \{1, \dots, d\}$, we denote $\{\lambda_{i, 1}, \dots, \lambda_{i, q+K}\}$ the 
dual basis for $\{N_1(x_i), \dots, N_{q+K}(x_i)\}$ constructed in \ref{sec:univariate_splines}. For any $j_1, \dots, j_d \in \{1, \dots, q + K\}$ and $f \in L_\infty([0, 1]^d)$, define
\[
\L_{j_1, \dots, j_d}f = \lambda_{1, j_1} \lambda_{2, j_2} \dots \lambda_{d, j_d}f.
\]
It is easy to see that $\left\{\L_{j_1, \dots, j_d} : j_1, \dots, j_d \in \{1, \dots, q + K\} \right\}$ form a dual basis for \newline $\left\{ N_{j_1}(x_1), \ldots, N_{j_d}(x_d)\right\}$. Now we formulate the approximation result from \cite[Theorem 12.5]{schumaker07}.
\begin{Prop}[\cite{schumaker07}, Theorem 12.5]
\label{thm:coef_bound}
Let $f \in L_\infty([0, 1]^d)$. Fix any $j_1, \dots, j_d \in \{1, \dots, q + K\}$ and let $\L_{j_1, \dots, j_d}$ be as defined above. Introduce
\[
A_{j_1, \dots, j_d} = \bigotimes\limits_{i=1}^d [ a_{j_i}, a_{j_i + q + 1} )\,.
\]
Then
\[
|\L_{j_1, \dots, j_d} f| \leq (2q + 1)^d 9^{d(q - 1)} \|f\|_{L_\infty(A_{j_1, \dots, j_d})}\,.
\]
\end{Prop}
Similarly to $\mathcal Q^{q, K}$, for any $f \in L_\infty([0, 1]^d)$, define a multivariate 
quasi-interpolant
\[
\mathcal Q_d^{q, K} f(x) = \sum\limits_{j_1, \dots, j_d = 1}^{q + K} (\L_{j_1, \dots, j_d} 
f) N_{j_1}(x_1) \dots N_{j_d}(x_d).
\]
Similarly to $N_j$'s, we omit the upper indices $q, K$ in the notation of $\mathcal Q_d$ when they are clear from context. We are ready to formulate the main result of this section.

\begin{Thm}
\label{thm:spline_approx}
Let $f \in \H^\beta([0, 1]^d, H)$ and fix a non-negative integer $\ell \leq \lbeta$. Let the linear operator $\mathcal Q_d^{\lbeta, K} : L_\infty([0,1]^d) \rightarrow \splinespace_d^{\lbeta, K}$ be as defined above. Then
\[
\left\|f - \mathcal Q_d^{\lbeta, K} f \right\|_{\H^\ell([0, 1]^d)} \leq  \frac{(\sqrt{2}e d)^{\beta} H}{K^{\beta-\ell}} + \frac{9^{d(\lbeta - 1)} (2\lbeta + 1)^{2d + \ell} (\sqrt{2}ed)^\beta H}{K^{\beta - \ell}}\,.
\]
\end{Thm}
\begin{proof}
For simplicity, we adopt the notation $q = \lbeta$. Fix a multi-index $\bgamma \in \mathbb N_0^d$, $|\bgamma| \leq \ell$. Then
\[
\|D^\bgamma f - D^\bgamma \mathcal Q_d f\|_{L_\infty([0, 1]^d)}
= \max\limits_{j_1, \dots, j_d \in \{1, \dots, 2q + K + 1\}} \|D^\bgamma f - D^\bgamma\mathcal Q_d f\|_{L_\infty(A_{j_1, \dots, j_d})}\,.
\]
Consider a polynomial of degree $q$
\[
P^f_{j_1, \dots, j_d}(x) = \sum\limits_{\balpha : |\balpha| \leq q} \frac{D^\balpha f(a_{j_1, \dots, j_d})}{\balpha!} (x - a_{j_1, \dots, j_d})^\balpha\,,
\]
where $a_{j_1, \dots, j_d} = (a_{j_1}, \dots, a_{j_d})$.
Here we used conventional notations $\balpha! = \alpha_1! \alpha_2! \dots \alpha_d!$ and	
\[
	(x - a_{j_1, \dots, j_d})^\balpha = (x_1 - a_{j_1})^{\alpha_1} (x_2 - a_{j_2})^{\alpha_2} \dots (x_d - a_{j_d})^{\alpha_d}.
\]
Since $P^f_{j_1, \dots, j_d}$ belongs to $\splinespace_d$, we have $\mathcal Q_d P^f_{j_1, \dots, j_d} = P^f_{j_1, \dots, j_d}$. By the triangle inequality,
\begin{equation}
\label{eq:derivative_bound}
\begin{split}
	\|D^\bgamma f - D^\bgamma \mathcal Q_d f\|_{L_\infty(A_{j_1, \dots, j_d})} 
	&
	\leq \|D^\bgamma (f - P^f_{j_1, \dots, j_d})\|_{L_\infty(A_{j_1, \dots, j_d})} \\
	&\quad
	+ \|D^\bgamma \mathcal Q_d (f - P^f_{j_1, \dots, j_d})\|_{L_\infty(A_{j_1, \dots, j_d})}.
\end{split}
\end{equation}
By Taylor's theorem and the definition of the H\"older class, the first term in the right-hand sight does not exceed
\begin{align*}
	&
	\|D^\bgamma f - D^\bgamma P^f_{j_1, \dots, j_d}\|_{L_\infty(A_{j_1, \dots, j_d})}
	\\&
	= \Bigg\| \sum\limits_{\balpha : |\balpha| = q - |\bgamma|} \frac{q (x - a_{j_1, \dots, j_d})^\balpha}{\balpha!} \int\limits_0^1 (1 - s)^{q-1}
	\\&\quad
	\cdot \left( D^{\balpha + \bgamma} f(sx + (1-s) a_{j_1, \dots, j_d}) - D^{\balpha + \bgamma} f(a_{j_1, \dots, j_d}) \right) \rmd s \Bigg\|_{L_\infty(A_{j_1, \dots, j_d})}
	\\&
	\leq \left\| \sum\limits_{\alpha : |\balpha| = q - |\bgamma|} \frac{q (q/K)^{q - |\bgamma|} H}{\balpha!} 
	\int\limits_{0}^{1} (1 - s)^{q-1} \| x - a_{j_1, \dots, j_d} \|^{\beta - q} \rmd s 
	\right\|_{L_\infty(A_{j_1, \dots, j_d})}
	\\&
	\leq \sum\limits_{\balpha : |\balpha| = q - |\bgamma|} \frac{(q/K)^{q - |\bgamma|} H}{\balpha!} (q\sqrt{d} /K)^{\beta - q}
	\leq \frac{(q d/K)^{\beta-|\bgamma|} H}{(q - |\bgamma|)!}.
\end{align*}
Here we used the fact that
\[
    \sum\limits_{\balpha \in \mathbb N_0^d : |\balpha| = m} \frac{m!}{\balpha!} = d^m \quad \text{for all $m \in \mathbb N_0$.}
\]
Moreover, note that, for any $r \in \{0, 1, \dots, q - 1\}$, it holds that
\[
    \frac{q^{\beta - r}}{(q - r)!}
    = \frac{q}{q - r}\cdot \frac{q^{\beta - r - 1}}{(q - r - 1)!}
    \geq \frac{q^{\beta - r - 1}}{(q - r - 1)!}.
\]
This yields that $q^{\beta - r} / (q - r)!$ achieves its maximum over $r \in \{0, 1, \dots, q\}$ at $r = 0$. Thus,
\[
    \|D^\bgamma f - D^\bgamma P^f_{j_1, \dots, j_d}\|_{L_\infty(A_{j_1, \dots, j_d})}
	\leq \frac{(q d/K)^{\beta-|\bgamma|} H}{(q - |\bgamma|)!}
	\leq \frac{q^\beta (d/K)^{\beta-|\bgamma|} H}{q!}.
\]
Since 
\begin{equation}
    \label{eq:factorial_lower}
    q! \geq \sqrt{2\pi} e^{-q}q^{q+1/2}
\end{equation}
for all $q \in \nset$, we obtain that
\[
    \|D^\bgamma f - D^\bgamma P^f_{j_1, \dots, j_d}\|_{L_\infty(A_{j_1, \dots, j_d})}
	\leq (\sqrt{2}e)^q (d/K)^{\beta-|\bgamma|} H.
\]
Now, we consider the second term in \eqref{eq:derivative_bound}, that is, $D^\bgamma \mathcal Q_d (f - P^f_{j_1, \dots, j_d})$. By the definition of  $\mathcal Q_d$, we have 
\[
	D^\bgamma \mathcal Q_d (f - P^f_{j_1, \dots, j_d})
	= \sum\limits_{i_1, \dots, i_d = 1}^{q + K} \L_{i_1, \dots, i_d} (f - P^f_{j_1, \dots, j_d}) D^\bgamma \left( N_{i_1}(x_1) \dots N_{i_d}(x_d) \right).
\]
The triangle inequality and \Cref{thm:coef_bound} implies that
\begin{align*}
	&
	\|D^\bgamma \mathcal Q_d (f - P^f_{j_1, \dots, j_d})\|_{L_\infty(A_{j_1, \dots, j_d})}
	\\&
	\leq \sum\limits_{i_1, \dots, i_d = 1}^{q + K} |\L_{i_1, \dots, i_d} (f - P^f_{j_1, \dots, j_d})| \left\| D^\bgamma N_{i_1}(x_1) \dots N_{i_d}(x_d) \right\|_{L_\infty(A_{j_1, \dots, j_d})}
	\\&
	\leq (2q + 1)^d 9^{d(q - 1)} \|f - P^f_{j_1, \dots, j_d}\|_{L_\infty(A_{j_1, \dots, j_d})} 
	\\&\quad
	\cdot \sum\limits_{i_1, \dots, i_d = 1}^{q + K}  \left\| D^\bgamma N_{i_1}(x_1) \dots N_{i_d}(x_d) \right\|_{L_\infty(A_{j_1, \dots, j_d})}.
\end{align*}
Again, by Taylor's theorem,
\[
	\|f - P^f_{j_1, \dots, j_d}\|_{L_\infty(A_{j_1, \dots, j_d})}
	\leq \frac{(q d/K)^{\beta} H}{q!}.
\]
By \Cref{prop:b-spline}, for any $m \in \{1, \dots, q+K\}$, $\text{supp}(N_{i_m}) = 
[a_{i_m}, a_{i_{m+q}}]$.
Hence, the intersection of $\text{supp}(N_{i_m})$ with $[a_{j_m}, a_{j_{m+q}}]$ has zero volume if $|i_m - j_m| > q$. This yields
\begin{align*}
	&
	\sum\limits_{i_1, \dots, i_d = 1}^{q + K}  \left\| D^\bgamma N_{i_1}(x_1) \dots 
	N_{i_d}(x_d) \right\|_{L_\infty(A_{j_1, \dots, j_d})}
	\\&
	= \sum\limits_{i_1, \dots, i_d = 1}^{q + K}  \left\| D^\bgamma N_{i_1}(x_1) \dots N_{i_d}(x_d) \right\|_{L_\infty(A_{j_1, \dots, j_d})} \prod_{m=1}^d \1\left( |i_m - j_m| \leq q \right).
\end{align*}
Applying \Cref{co:b-spline_deriv}, we obtain that
\begin{align*}
	&
	\|D^\bgamma \mathcal Q_d (f - P^f_{j_1, \dots, j_d}) \|_{L_\infty(A_{j_1, \dots, j_d})}
	\\&
	\leq \frac{(2q + 1)^d 9^{d(q - 1)} (q d/K)^{\beta} H}{q!}
	\sum\limits_{i_1, \dots, i_d = 1}^{q + K}   
	\prod_{m=1}^d \frac{q! (2K)^{\gamma_m}}{(q - \gamma_m)!} \1\left( |i_m - j_m| \leq q \right)
	\\&
	\leq \frac{(2q + 1)^d 9^{d(q - 1)} (q d/K)^{\beta} H}{q!}
	\sum\limits_{i_1, \dots, i_d = 1}^{q + K}   
	(2Kq)^{|\bgamma|} \prod_{m=1}^d \1\left( |i_m - j_m| \leq q \right).
\end{align*}
The sum in the right hand side contains at most $(2q + 1)^d$ non-zero terms. Hence,
\begin{align*}
	\|D^\bgamma \mathcal Q_d (f - P^f_{j_1, \dots, j_d})\|_{L_\infty(A_{j_1, \dots, j_d})}
	&
	= \frac{(2q + 1)^{2d} 9^{d(q - 1)} (qd/K)^{\beta} (2Kq)^{|\bgamma|} H}{q!}
	\\&
	\leq \frac{9^{d(q - 1)} (2q + 1)^{2d + |\bgamma|} (\sqrt{2}ed)^\beta H}{K^{\beta - |\bgamma|}}\,,
\end{align*}
where we used \eqref{eq:factorial_lower}.

\end{proof}

\end{document}